\documentclass[11pt]{amsart}
\usepackage{amsmath, amssymb}
\newfont {\cyr} {wncyr10}
\renewcommand{\labelenumi}{{(\roman{enumi})}}

\usepackage{amsfonts}
\usepackage{amsmath}
\usepackage{amssymb}
\usepackage{setspace}
\usepackage{multicol}
\usepackage{color}

\mathchardef\tnode="020E

\def\arc{ 
 \hbox{\kern -0.15em \vbox{\hrule width 2.5em height 0.6ex depth -0.5 ex} \kern
-0.33em}}

\def\darc{
 \rlap{\lower0.2ex\arc}{\raise0.2ex\arc}}

\def\tarc{
 \rlap{\rlap{\lower0.4ex\arc}{\raise0.4ex\arc}}{\arc}}

\def\stroke#1{
 \kern 0.05
\rlap\arc{{\textstyle{#1}}\atop\phantom\arc} \kern -0.22em}

\def\dstroke#1{
 \kern 0.05em
\rlap\darc{{\textstyle{#1}}\atop\phantom\darc} \kern -0.22em}

\def\centerscript#1{
 \setbox0=\hbox{$\tnode$} \hbox to
\wd0{\hss$\scriptstyle{#1}$\hss}}

\def\node{
 \def\super{} \def\sub{}
\futurelet\next\dolabellednode}

  \let\sp=^ \let\sb=_

  \def\dolabellednode{%
   \ifx\next\sb\let\next\getsub \else \ifx\next\sp\let\next\getsuper
\else\let\next\donode \fi \fi \next}

  \def\getsub_#1{\def\sub{#1}\futurelet\next\dolabellednode}

\def\getsuper^#1{\def\super{#1}\futurelet\next \dolabellednode}

  \def\donode{%
   \rlap{$\mathop{\phantom\tnode}\limits_{\centerscript{\sub}}
^{\centerscript{\super}}$}\tnode}

\def\varcdn{
 \kern
-0.03em\vbox{\kern -0.5ex \hbox to \wd0{\hss\vrule width 0.04em depth 5.8ex\hss} \kern -0.3ex \hbox{$\tnode$}}}

\def\a3{\node_1\arc\node_2\arc\node_3}

\def\c3{\node_1\arc\node_2\darc\node_3}

\def\m24{\node\arc\node\dstroke{\sim}\node} \def\u43{\node\darc\node\dstroke{\sim}\node}

\newcommand{\varcdnl}[1]{ 
\kern -0.03em\vbox{\kern -0.5ex \hbox to \wd0{\hss\vrule width 0.04em depth 5.8ex\hss} \kern -0.3ex
\hbox{$\tnode^{#1}$}}}

\def\d4{\node^1\arc\node^{2}_{\varcdnl{3}}\arc\node^{4}}

\def\nodef{
\def\super{} \def\sub{} \futurelet\next\dolabellednodef}

  \let\sp=^ \let\sb=_

  \def\dolabellednodef{%
  \ifx\next\sb\let\next\getsubf \else
\ifx\next\sp\let\next\getsuperf \else\let\next\donodef \fi \fi \next}

\def\getsubf_#1{\def\sub{#1}\futurelet\next\dolabellednodef}

\def\getsuperf^#1{\def\super{#1}\futurelet\next \dolabellednodef}

  \def\donodef{%
  \rlap{$\mathop{\phantom\tnodef}\limits_{\centerscript{\sub}}
^{\centerscript{\super}}$}\tnodef}

\def\varcdnf{
 \kern -0.03em\vbox{\kern -0.5ex \hbox to \wd0{\hss\vrule width 0.04em depth
5.8ex\hss} \kern -0.3ex \hbox{$\tnodef$}}}

\newtheorem{theorem}{Theorem}[section]
\newtheorem{lemma}[theorem]{Lemma}

\newtheorem{definition}[theorem]{Definition}

\newtheorem{hyp}[theorem]{Hypothesis}

\newcounter{claim}[theorem]

\newcounter{cclaim}[theorem]

\def \udot {{}^{\textstyle .}} 

\newcommand{\F}{\mathrm{F}}
 \newcommand{\G}{\mathrm{G}}  
\newcommand{\Q}{\mathrm{Q}}  
\newcommand{\Aut}{\mathrm{Aut}}  
  
\newcommand{\Syl}{\mathrm{Syl}}\newcommand{\syl}{\mathrm{Syl}} 
  \newcommand{\GF}{\mathrm{GF}}
  \newcommand{\SL}{\mathrm{SL}}
  
\newcommand{\PSp}{\mathrm{PSp}} \newcommand{\Sym}{\mathrm{Sym}}
\newcommand{\Alt}{\mathrm{Alt}}  
  
\newcommand{\U}{\mathrm{PSU}}

\def \syl {\hbox {\rm Syl}}\def \Syl {\hbox {\rm Syl}}

\def \CC {\mathcal C}

\def \Aut{ \mathrm {Aut}}

\def \Mat{\mbox {\rm Mat}}

\def \Co {\mbox {\rm Co}}

\def \McL{\mbox {\rm McL}}

\def \PSU {\mbox {\rm \PSU}}

\begin{document}
\renewcommand{\labelenumi}{(\roman{enumi})}

\title  {An improved $3$-local characterisation of $\McL$ and its automorphism group}
 \author{Chris Parker}
  \author{Gernot Stroth}
\address{Chris Parker\\
School of Mathematics\\
University of Birmingham\\
Edgbaston\\
Birmingham B15 2TT\\
United Kingdom} \email{c.w.parker@bham.ac.uk}

\address{Gernot Stroth\\
Institut f\"ur Mathematik\\ Universit\"at Halle - Wittenberg\\
Theodor Lieser Str. 5\\ 06099 Halle\\ Germany}
\email{gernot.stroth@mathematik.uni-halle.de}

\email {}

\def\l {\lambda}
\def \eps {\varepsilon}
\def \Irr {\mathrm {Irr}}
\date{\today}

\maketitle \pagestyle{myheadings}

\markright{{\sc }} \markleft{{\sc Chris Parker and Gernot Stroth}}
\begin{abstract}
This article presents a $3$-local characterisation of the sporadic simple group $\McL$ and its automorphism group.  The proof of the theorem is underpinned by two further identification theorems, one due to Camina and Collins and the other proved in this paper. Both  these supporting results are proved by using  character theoretic methods. The main theorem is applied in  our  investigation of groups with a large $3$-subgroup \cite{PS1}.
\end{abstract}
\section{Introduction}

This article extends earlier work of Parker and Rowley \cite{McL} in which the McLaughlin sporadic simple group $\McL$ and its automorphism group are characterised by certain 3-local information. Suppose that $p$ is a prime and $G$ is a finite group. Then the normalizer of a non-trivial $p$-subgroup of $G$  is called a \emph{$p$-local subgroup} of $G$.  A subgroup $M$ of $G$ is said to be of \emph{characteristic} $p$ provided $F^*(M) = O_p(M)$ where $F^*(M)$ is the generalized Fitting subgroup of $M$.  See \cite{Aschbacher} for the fundamental properties of the generalized Fitting subgroup.  The group $G$ is of \emph{local characteristic $p$} if every $p$-local subgroup of $G$ has characteristic $p$ and $G$ is of \emph{parabolic characteristic $p$} if every $p$-local subgroup of $G$ which contains a Sylow $p$-subgroup of $G$ has characteristic $p$. The difference between these two group theoretic properties is  the difference between  the characterisation theorem presented in \cite{McL} and the theorem presented in this article. The main theorem of the former article essentially assumes that the group under investigation is of local characteristic $3$.  The theorem we prove here in  essence only assumes that the group $G$ has parabolic characteristic $3$ though it is not necessary to articulate this explicitly in the statement of the  theorem.

 \begin{theorem}\label{McL}  Suppose that G is a finite group, $S \in \Syl_3
(G)$, $Z = Z(S)$ and $J$ is an elementary
abelian subgroup of $S$ of order $3^4$.  Further assume that
\begin{enumerate}
\item  $O^{3'}(N_G(Z)) \approx 3^{1+4}_+.2\udot\Alt(5)$;
\item  $O^{3^\prime}(N_G(J)) \approx 3^4.\Alt(6)$; and
\item $C_G(O_3(C_G(Z)))\le O_3(C_G(Z))$.
\end{enumerate}
Then $G \cong \McL$ or $\Aut(\McL)$.
 \end{theorem}

The main theorem of \cite{McL} carries the additional hypothesis $$C_G(O_3(C_G(x))) \leq O_3(C_G(x))$$ for all $x \in J^\#$. Hence to prove Theorem~\ref{McL}, we just need to show that this inequality  is a consequence of the assumptions in Theorem~\ref{McL}.

 Theorem~\ref{McL} is applied in our investigation of exceptional cases which arise in the determination of  groups with a large $p$-subgroup \cite{PS1}.
A  $p$-subgroup $Q$ of a group $G$ is \emph{large} if and only if
\medskip
\begin{enumerate}
\item[(L1)]\label{1} $F^*(N_G(Q))=Q$; and
\item[(L2)]\label{2} for all  non-trivial subgroups $U$ of $ Z(Q)$, we have  $N_G(U)\le N_G(Q)$.
\end{enumerate}
\medskip
It is an elementary observation that most of the groups of  Lie type in characteristic $p$ have a  \emph{large}
$p$-subgroup. The only Lie type groups in characteristic $p$ and rank at least $2$ which do not contain
such a subgroup are $\PSp_{2n}(2^a)$, $\F_4(2^a)$ and $\G_2(3^a)$. It is not difficult to show that groups $G$
which contain a large $p$-subgroup are of parabolic characteristic $p$ (see \cite[Lemma 2.1]{PS1}).  The work in
\cite{MSS} begins the determination of the structure of the $p$-local overgroups of $S$ which are not contained in $N_G(Q)$. The idea is to collect data about
the $p$-local subgroups of $G$ which contain a fixed Sylow $p$-subgroup and then using this  information  show that the subgroup generated by them is  a  group of Lie type.   However  sometimes one is confronted with the following situation:
some (but perhaps not all) of the  $p$-local subgroups of $G$ containing a given Sylow $p$-subgroup $S$ of $G$
generate a subgroup   $H$ and $F^*(H)$  is known to be isomorphic to a Lie type group in characteristic $p$. Usually  $G=H$. To show this, we assume that  $H$ is a proper subgroup of $G$,   and first establish  that $H$ contains \emph{all} the  $p$-local subgroups of $G$ which contain $S$. The next step then demonstrates that $H$ is strongly $p$-embedded in $G$ at which stage \cite{PSStrong} is applicable and delivers $G =
H$. The last two steps are reasonably well understood, at least for groups with mild extra assumptions imposed. However
it might be that the first step cannot be made.  Typically this occurs only when $N_G(Q)$ is not contained in
$H$. The main theorem of this article is applied in just this type of situation. Specifically, it  is  applied in the case that   $p=3$ and $F^\ast(H)$ is the group $\U_4(3)$. In this case in $F^\ast(H)$  the large $3$-subgroup $Q$ is extraspecial of order $3^5$ and is the
largest normal 3-subgroup in the normalizer of a root group in $F^\ast(H)$. In this configuration  we are not able to show that $N_G(Q) \le H$, as is demonstrated by noting that  $\U_4(3)$ is a subgroup of $\McL$. In fact in \cite{PS1} we show that, if
$F^\ast(H) \cong \U_4(3)$ and $N_G(Q) \not\leq H$, then $F^\ast(G) \cong \McL, \Co_2$, or $\U_6(2)$. It is precisely for the identification of  $\McL$ that we need the result of this paper.

The route to prove Theorem~\ref{McL} is paved by two preliminary  results. These theorems state that under certain hypotheses a subgroup $H$ of a group $G$ is actually equal to $G$.

We denote by  $K$  the  subgroup of $\Alt(9)$ which normalizes $$J= \langle (1,2,3), (4,5,6), (7,8,9)\rangle.$$ Thus $$K=  \langle (1,2,3),   (1,4,7)(2,5,8)(3,6,9), (1,2)(4,5),
(1,4)(2,5)(3,6)(7,8)\rangle.$$
 The first of the preliminary theorems  is as follows:

\begin{theorem}\label{G=H} Suppose $G$ is a finite group, $H \le G$  with $H \cong K$ and  $J= O_3(H)$.   If  $C_G(j) \le H$ for all $j \in J^\#$  and $J$ is strongly closed in $H$ with respect to $G$, then $G=H$.
\end{theorem}

The second theorem that we require is due to  Camina and Collins \cite[Proposition 5]{caminacollins} and  we record it here.

\begin{theorem}[Camina and Collins \cite{caminacollins}]\label{108} Suppose  $G$ is a finite group and $H$ is a subgroup of $G$ with $$H \cong \langle (1,2,3), (4,5,6),(7,8,9), (1,2)(4,5), (1,2)(7,8)\rangle \approx 3^3:2^2.$$ Let $T \in \Syl_3(H)$ and assume $N_G(T)= H$ and $C_G(t) \le H$ for all $t \in T^\#$. Then $G=H$.
\end{theorem}

The hypothesis about the embedding of $H$ in $G$ in Theorem~\ref{108} is equivalent  to saying that $H$ is strongly $3$-embedded in $G$.
The proofs of both of the above theorems  exploit the methods introduced by Suzuki which permit  parts of the character table of  $G$ to be constructed. Details of the theory behind the Suzuki method are very well presented in \cite{CurtisReiner}.  The embedding properties in both Theorems~\ref{G=H} and \ref{108}  which assert that centralizers of certain  elements of $H$ are contained in $H$ are precisely the requirements needed to make the Suzuki theory of special classes work. The result of the  Suzuki method  are fragments of possible character tables for $G$. These fragmentary  tables provide all the character values on certain elements of order $3$ in $H$.

For the proof of Theorem~\ref{G=H},  as we start to build up the required character decompositions there is an overwhelming number of possibilities and so we have performed this calculation using {\sc Magma} \cite{Magma}.  The result of this computation is four fragments of possible  character tables for $G$. However  this statement is rather disingenuous as in fact each fragment represents many possible character tables as the entries of the partial tables are only known up to sign choices.

 Recall that   for $x,y, z \in G$ the $G$-structure constant $$a_{xyz}= |\{(a,b) \in x^G \times y^G \mid ab=z\}|$$ is determined by the character table of $G$ by the following equation
$$a_{xyz} = \frac{|G|}{|C_G(x)||C_G(y)|}\sum_{\chi \in \Irr(G)}
\frac{\chi (x)\chi (y)\chi (z^{-1})}{\chi (1)}.$$
If we select $x, y \in H$ such that $C_G(x)$ and $C_G(y)$ are contained in $H$, then we know $|C_G(x)|=|C_H(x)| $  and $|C_G(y)|= |C_H(y)|$. Furthermore, for certain choices of $x$, $y$ and $z$, we know all the character values of $x$, $y$ and $z$. Thus the only unknown quantity on the left hand side of the structure constant equation is $|G|$.
The proof of Theorem~\ref{G=H} pivots    on  the following fundamental fact \cite{FTh}  about groups generated by two elements of order $3$ which have product of order $3$: \begin{quote}\emph{Suppose that $X= \langle x,y \mid x^3=y^3=(xy)^3=1\rangle$. Then $X$ has an abelian normal subgroup of index $3$.}
\end{quote}
This fact allows us  to show that for certain $z\in H$, and  for certain pairs $x,y $ of elements of order $3$ in $G$, if $xy=z$ then $x,y \in H$. This means that we can calculate $a_{xyz}$ in $H$ and so we know the left hand side of the structure constant formula. Hence in principle we can determine $|G|$. What in fact happens is that we can discover enough information about $|G|$ to decide that the possible partial character tables are invalid or to show that $G=H$.

Once Theorem \ref{G=H} is proved, in Section 3 we prove  Theorem~\ref{McL}.  Now suppose  $G$ and $J$ are as in Theorem~\ref{McL} and set $Q= O_3(N_G(Z))$ and $M= N_G(J)$.  The initial part of the proof recalls some pertinent facts from \cite{McL}. In particular, we recall that $M/J \cong \Mat(10)$ or $2 \times \Mat(10)$. For $y \in Q\setminus Z$, we also show that $O^3(C_M(y))/\langle y \rangle$ is isomorphic to the  group $H$ in Theorem~\ref{108} if   $N_G(J)/J \cong \Mat(10)$ and to the group $H $ in Theorem~\ref{G=H} if  $N_G(J)/J \cong 2 \times \Mat(10)$. The main technical result in this section proves, for $x \in O_3(C_M(y)/\langle y\rangle)$, $C_{C_G(y)/\langle y\rangle}(x) \le C_M(y)/\langle y\rangle$ and exploits the theorem of Smith and Tyrer \cite{SmTy}.  Once this is proved we   quickly finish the proof of Theorem~\ref{McL} with the help of Theorems~\ref{G=H} and \ref{108}.

Our notation follows that of \cite{Aschbacher} and \cite{Gor}.
We use {\sc Atlas} \cite{Atlas} notation for group extensions.   For odd primes $p$, the extraspecial groups of exponent $p$ and
order $p^{2n+1}$ are denoted by $p^{1+2n}_+$.  The quaternion group of order $8$ is   $\Q_8$ and $\Mat(10)$ is the Mathieu group of degree $10$. A  central product of groups $H$ and $K$ will be denoted $H\circ K$. For a subset $X$ of a
group $G$, $X^G$ is the  set of $G$-conjugates of $X$.   From time to time  we shall give suggestive descriptions of groups which indicate the isomorphism type of certain composition factors. We refer to such descriptions as the \emph{shape} of a group. Groups of the same shape have normal series with isomorphic sections. We use the symbol $\approx$ to indicate the shape of a group. All the groups in this paper are finite groups.

\bigskip

\noindent {\bf Acknowledgement.}  The first author is  grateful to the DFG for their support and thanks the mathematics department in Halle for their hospitality. Both authors would like to thank Michael Collins for drawing their attention to his paper with Alan Camina in which Theorem~\ref{108} is proved.

\section{Proof of Theorem~\ref{G=H}}

In this section we use the Suzuki method which exploits virtual characters to prove Theorem~ \ref{G=H}.
We recall the following definition from \cite[Definition 14.5]{CurtisReiner}.

\begin{definition}[Suzuki Special Classes]\label{sc} Let $G$ be a group and $H$ be a subgroup of $G$. Suppose that
$\mathcal C= \bigcup_{i=1}^n \mathcal C_i$ is a union of conjugacy
classes of $H$. Then $\mathcal C$ is called a set of special classes
in $H$ provided the following three conditions hold.
\begin{enumerate} \item $C_G(h) \le H$ for all $h \in \mathcal C$;
\item $\mathcal C_i^G\cap  \mathcal C = \mathcal C_i$ for $1\le i\le
n$; and
\item if $h \in \mathcal C$ and $\langle h\rangle= \langle f\rangle$, then $f\in \mathcal
C$.
\end{enumerate}
\end{definition}
 As mentioned in the introduction, the Suzuki method and the theory behind it are well explained in \cite{CurtisReiner} and   we refer the reader explicitly to Section 14B in \cite{CurtisReiner}.

Let $K$ be the  subgroup of $\Alt(9)$ which normalizes $$J= \langle (1,2,3), (4,5,6), (7,8,9)\rangle.$$ Thus $$K= \langle (1,2,3),   (1,4,7)(2,5,8)(3,6,9), (1,2)(4,5),
(1,4)(2,5)(3,6)(7,8)\rangle.$$
We remark that $K$ has shape $3^3{:}\Sym(4)$ but is not the unique group of this shape which has characteristic $3$.
We assume that $G$, $H$  and $J$ are as in the statement of Theorem~\ref{G=H}.  Hence we may identify  $H$ with $K$  and for all $j \in J^\#$  we have  $C_G(j) \le H$. Furthermore, $J$ is strongly closed in $H$ with respect to $G$. We recall that this means that $J^g \cap H \le J$ for all $g \in G$.

  We have used \cite{Magma} to produce the character table of $H$ and  have presented the result in  Table~\ref{Tab}.
\begin{table}$$
\begin{tabular}{c|ccccccccccccccc}
Class&$\CC_1$&$\CC_2$&$\CC_3$&$\CC_4$&$\CC_5$&$\CC_6$&$\CC_7$&$\CC_8$&$\CC_9$&$\CC_{10}$&$\CC_{11}$&
$\CC_{12}$&$\CC_{13}$&$\CC_{14}$\\
Size& 1 &27 &54  &6  &8 &12 &72 &54 &54 &108 &72 &72 & 54  &54\\
Order &1&2&2&3&3&3&3&4&6&6&9&9&12&12\\\hline
$\chi_1$ &1&1&1&1&1&1&1&1&1&1&1&1&1&1\\$\chi_2$ &1&1&-1&1&1&1&1&-1&1&-1&1&1&-1&-1\\$\chi_3$ &2&2&0&2&2&2&-1&0&2&0&-1&-1&0&0
\\$\chi_4$
&3&-1&-1&3&3&3&0&1&-1&-1&0&0&1&1\\$\chi_5$& 3&-1&1&3&3&3&0&-1&-1&1&0&0&-1&-1\\$\chi_6$ &6&2&0&3&-3&0&0&2&-1&0&0&0&-1&-1
\\$\chi_7$ &6&2&0&3&-3&0&0&-2&-1&0&0&0&1&1\\$\chi_8$ &6&-2&0&3&-3&0&0&0&1&0&0&0&$\zeta$&-$\zeta$\\$\chi_9$ &6&-2&0&3&-3&0&0&0&1&0&0&0&-$\zeta$
&$\zeta$\\$\chi_{10}$ &8&0&0&-4&-1&2&2&0&0&0&-1&-1&0&0\\$\chi_{11}$ &8&0&0&-4&-1&2&-1&0&0&0&2&-1&0&0\\$\chi_{12}$
&8&0&0&-4&-1&2&-1&0&0&0&-1& 2&0&0\\$\chi_{13}$ &12&0&-2&0&3&-3&0&0&0&1&0&0&0&0\\$\chi_{14 }$ &12&0&2&0&3&-3&0&0&0&-1&0&0&0&0
\end{tabular}
$$
\caption{The character table of $H$. Here $\zeta=\sqrt 3$.
}\label{Tab}
\end{table}
The conjugacy classes of $H$ will be represented by $\mathcal C_1, \dots, \mathcal C_{14}$ as labeled in Table~\ref{Tab}.
We let $x= (1,2,3) \in \CC_4$,  $y = (1,2,3)(4,5,6)\in \CC_6$,  $z= (1,2,3)(4,5,6)(7,8,9)\in \CC_5$  and $\mathcal J = x^H\cup y^H\cup z^H$.  So $$\mathcal J =  \mathcal C_4\cup \mathcal C_6\cup \mathcal C_5= J^\#.$$
Notice that $z$ is $3$-central and so we may suppose that $T$ is chosen so that $Z(T)= \langle z\rangle$.

\begin{lemma}\label{special1} The following hold:
\begin{enumerate}
\item $T \in \syl_3(G)$;
\item $H$ controls $G$-fusion of elements of order $3$ in $T$;
\item $ \mathcal C=\mathcal J \cup \bigcup_{i=9}^{14} \mathcal C_i$ is a set of special classes in $H$; and
\item if $t \in \mathcal J$ and $a, b \in \mathcal J^G$ satisfy $ab=t$, then either $\langle a \rangle $, $\langle b \rangle $ and $\langle t \rangle$ are all $G$-conjugate or $a,b  \in J^\#$.
    \end{enumerate}
\end{lemma}

\begin{proof} We calculate $Z(T) \le J$ and $Z(T)= \langle z \rangle$. Hence $C_G(Z(T)) \le H$ by hypothesis and so $T \in \syl_3(G)$ which is (i).

Using Table~\ref{Tab} and the fact that $x,y,z \in J$, we have
$|C_H(x)|=|C_G(x)|= 108$, $|C_H(y)|=|C_G(y)|= 54$ and $|C_H(z)|=|C_G(z)|=81$. Hence $x$, $y$ and $z$ are not conjugate in $G$. Let $w \in T\cap \mathcal C_7$. Then $w$ has order $3$.  As, by assumption, $J$ is strongly closed in $H$,  $w$ is not $G$-conjugate to some element in $J$.
This completes the proof of (ii).

For $s \in  \bigcup_{i=9}^{14} \mathcal C_i$, some power of $s$ is contained in $\mathcal J$. Hence parts (i) and (ii) and the hypothesis that $C_G(t ) \le H$ for $t \in\mathcal J$  imply that (i) and (iii) of Definition~\ref{sc} hold. Furthermore, as $C_G(s) \le H$ for all $s \in \mathcal C$, the  only possibility for Definition~\ref{sc} (ii) to fail is if classes $\mathcal C_{11}$ and $\mathcal C_{12}$ fuse in $G$ or if $\mathcal C_{13}$ and $\mathcal C_{14}$ fuse in $G$. In the first case  there are $a \in \mathcal C_{12}$ and $b \in \mathcal C_{13}$ such that $\langle a \rangle = \langle b \rangle$ and $a^3= b^3 \in \mathcal C_5$ is $3$-central. Hence if these classes fuse, they do so in the centralizer of $a^3$ and thus in $H$ which is a contradiction. A similar argument shows  classes $\mathcal C_{13}$ and $\mathcal C_{14}$ also cannot fuse. Hence $\mathcal C$ is a set of special classes in $H$.

 Before we start with the proof of (iv) we make some remarks. Let $w \in \mathcal C_7$. Then we may suppose that   $T=J\langle w \rangle$. From the structure
of $H$ (or Table~\ref{Tab}),  we know  $C_H(w)$ is $3$-closed with Sylow $3$-subgroup $\langle z,w\rangle$. In particular, $\langle z,w\rangle$ contains exactly two conjugates of $z$ and therefore $N_G(\langle z,w\rangle) \le C_G(z) \le H$ which implies

\medskip
\begin{claim}\label{xtar} For $w \in \mathcal C_7\cap T$, $\langle z, w \rangle \in \syl_3(C_G(w)) $  and $\langle z, w \rangle  $ contains two $G$-conjugates of $z$ and six  $G$-conjugates of $w$.\end{claim}
\medskip

 Suppose that   $t \in \mathcal J$ and $a, b \in \mathcal J^G$ satisfy $ab=t$.  To prove (iv) we may suppose that $\langle a\rangle$, $\langle b\rangle$ and $\langle t \rangle$ are not  all  $G$-conjugate. Furthermore, setting $X= \langle a, b \rangle$ we may suppose that $X= A \langle a \rangle$ where $A$ is a normal abelian subgroup of $X$.  Let $A_3$ be the Sylow $3$-subgroup of $A$.    As $\langle a\rangle$, $\langle b\rangle$ and $\langle t \rangle$ are not  all  $G$-conjugate, $A_3\neq 1$.

Assume that $X \not \le H$.  We first  suppose that $t \in A_3$. Then $A_3\le A \le C_G(t) \le H$ and hence $a \not \in H$. Let $B = C_{A_3}(a)$. If $B \cap J \neq 1$, then $a \in C_G(B\cap J) \le H$, which is a contradiction. Thus $B$ has  order $3$ and the non-trivial elements of $B$ are in class $\mathcal C_7$. Since $A_3$ is abelian and $t \in A_3$, we now have $\langle B, t \rangle\le  A_3 \le C_H(B)$. As $|C_H(B)|_3=9$ by \ref{xtar}, $ \langle B, t \rangle= A_3$ and $\langle t \rangle = A_3 \cap J$ is the unique $G$-conjugate of $\langle t \rangle$ contained in $A_3$. But then $a\in C_G(t) \le H$  a contradiction.  Hence $t \not \in A_3$ and $X = \langle t \rangle A$.
If $C_{A_3}(t) \cap J \not= 1$, then we have that $X = \langle t \rangle A \leq C_G(C_{A_3}(t) \cap J) \leq H$, a contradiction. So we have that $C_{A_3}(t)$ is of order three and the non-trivial elements are all in $\mathcal C_7$. By
\ref{xtar}
we have that $C_{A_3}(t)\langle t \rangle$ is a Sylow 3-subgroup of $C_G(C_{A_3}(t))$ and so $ U_t = C_{A_3}(t)\langle t \rangle$ is a Sylow 3-subgroup of $X$. As $a$ and $b$ are not in $\mathcal C_7^G$ by (ii), we have that also $a$ and $b$ are not in $A$ and so $U_a = \langle   C_{A_3}(t), a \rangle$, $U_b = \langle C_{A_3}(t), b \rangle$ are Sylow 3-subgroups of $X$ too. Now we have that $\mathcal J^G \cap \langle t, C_{A_3}(t) \rangle = \{t,t^{-1}\}$ by (ii). As $U_t$, $U_b$ and $U_a$ are conjugate in $X$, we have that also $|U_a \cap \mathcal J^G| = 2 = |U_b \cap \mathcal J^G|$.  However this means that $\langle t \rangle$, $ \langle a \rangle$ and $ \langle b \rangle$ are all $X$-conjugate, which is a contradiction. Therefore (iv) holds.
\end{proof}

For $g,h,k\in G$ we let $a_{ghk}$ be the corresponding $G$-structure  constant. Recall $x= (1,2,3)$,   $y = (1,2,3)(4,5,6)$ and  $z= (1,2,3)(4,5,6)(7,8,9)$.

\begin{lemma}\label{structconst} We have the following $G$-structure constants values: $a_{xyz}= 3$, $a_{xxy} =2$, $a_{xxz}=0$, $a_{yyx}=4$, $a_{yyz}=6$, $a_{zzx}=4$, $a_{zzy}=2$.
\end{lemma}

\begin{proof} Because of Lemma~\ref{special1} (iv) we can calculate these $G$-structure constants in $H$. To do this we may either calculate by hand in $K$ or use the character table of $H$   presented in Table~\ref{Tab}.
\end{proof}

\begin{lemma}\label{index} $|G:H|\equiv 1 \pmod {27}$.
\end{lemma}

\begin{proof} Suppose that $x \in G$ and consider $J \cap H^x$.  As $J$ is strongly closed in $H$ with respect to $G$,  $J \cap H^x \le J^x$. Hence, if $J \cap H^x > 1$, then by the assumption on the centralizers of elements in $J^\#$,  $J \le H^x$ and consequently $J^x= J$ and $x \in N_G(J)$. Thus the conjugation action of $J$ on $\{J^x\mid  x \in G\}$  fixes $J$ and  otherwise has orbits of length $|J|= 27$. Since $$N_G(J) = C_{N_G(J)}(z)H=  C_{H}(z)H= H$$ by Lemma~\ref{special1} (ii),   we conclude that $|G:H|= |\{J^x\mid  x \in G\}| \equiv 1 \pmod {27}$ as claimed.

\end{proof}

We use the notation for characters of $H$ introduced in Table~\ref{Tab}.
\begin{lemma}\label{mod9} Suppose that  $\theta$ is a virtual character of $G$. Then $\theta(1) \equiv \theta(z) \pmod 9$.
\end{lemma}

\begin{proof} Set $\psi = \sum_{i=6}^{14} \chi_i$. We have $\psi(1) = 72$, $\psi(z)= -9$ and $\psi(w)= 0$ if $w$ is not conjugate to either $1$ or $z$. Hence, as there are eight conjugates of $z$ in $H$, $$(\psi,\theta_H)=\frac{1}{|H|}(72\theta(1)-72\theta(z))= \frac{1}{9}(\theta(1)-\theta(z))$$ is an integer. This proves the claim.
\end{proof}

\begin{lemma}\label{mod27} Suppose that  $\theta$ is a virtual character of $G$. Then $$\theta(1) \equiv 4\theta(z)-3\theta(x) \pmod {27}.$$
\end{lemma}

\begin{proof} Set  $\psi = \sum_{i=6}^{9} \chi_i$. We have $\psi(1) = 24$, $\psi(x)= 12$, $\psi(z)= -12$ and $\psi(w)= 0$ if $w$ is not conjugate to either $1$ or $z$. Hence $$(\psi,\theta_H)_H=\frac{1}{|H|}(24\theta(1)+6\cdot 12\theta(x)-8\cdot 12 \theta(z))= \frac{1}{27}(\theta(1)+3\theta(x) -4\theta(z))$$ is an integer which proves the claim.
\end{proof}

We now follow the Suzuki method.  We  find the following basis for the class functions which vanish off the special classes $\mathcal C$:
\begin{eqnarray*}
\lambda_1&=& 2\chi_1+\chi_{13}-\chi_6-\chi_{10},\\
\lambda_2&=& \chi_1-\chi_5-\chi_6+\chi_{12},\\
\lambda_3 &=& \chi_1+\chi_2+\chi_8-\chi_{10},\\
\lambda_4&=&\chi_1+\chi_5+\chi_{12}-\chi_{14},\\
\lambda_5&=& \chi_2 -\chi_4-\chi_7+\chi_{12},\\
\lambda_6&=& \chi_{12}-\chi_3-\chi_8,\\
\lambda_7&=& \chi_8-\chi_9,\\
\lambda_8&=& \chi_{11}-\chi_{12} \text { and}\\
\lambda_9&=& \chi_4+\chi_5-\chi_8.
\end{eqnarray*}

For $1\le i\le 9$, define $$\gamma_i = \sum_{j=1}^{15} \chi_j(t_i)\chi_j$$  where  $t_i \in \mathcal{C}_{i+3}$ for $1 \leq i \leq 3$ and $t_i \in \mathcal{C}_{i+5}$ for $4  \leq i \leq 9$.  By the second orthogonality
relations, for $1\le i\le 9$, $$\gamma_i\in \langle \lambda_j\mid 1 \le j \leq 9\rangle.$$  In the matrix $D$ below,
row $i$ describes the decomposition of $\gamma_i$ in terms of the $\lambda_j$,  $1 \le j \le 9$.
{\small $$D=\left( \begin{array}{rrrrrrrrr}
0&-3&4&0&-3&-2&-3&-4&0\\
3&0&-2&-3&3&-2&3&-1&6\\
-3&3&1&3&0&-2&0&2&3\\
0&1&0&0&1&-2&-1&0&0\\
1&-1&-1&1&0&0&0&0&-1\\
0&0&1&0&0&1&0&2&0\\
0&0&1&0&0&1&0&-1&0\\
0&1&0&0&-1&0&\zeta&0&0\\
0&1&0&0&-1&0&-\zeta&0&0
\end{array} \right).$$}

We now determine the  possibilities for the matrices $B$ which have columns indexed by the irreducible characters $\Irr(G)=\{\theta_1, \dots, \theta_s\} $ of $G$, rows indexed by the induced virtual characters $\mu_i = \lambda_i^G$ and $(i,j)th$ entry indicating  the multiplicity of the character $\theta_j$ in the virtual character $\mu_i$. To determine the entries in the candidates for $B$, we use that, for $1\le i \le k \le 9$, $$(\mu_i,\mu_k)= (\lambda_i, \lambda_k)$$  and that the multiplicity of the principal character of $H$ in $\lambda_i$ is the same as the multiplicity of the principal character of $G$ in $\mu_i$
(see \cite[Lemma 14.9, Theorem  14.11]{CurtisReiner}). The calculation to determine the candidates for $B$ was performed using {\sc{Magma}} \cite{Magma}.  From the candidate matrices $B$, we can calculate fragments of the character table of $G$ by calculating $C=(DB)^t$. The columns of the fragments  are indexed by the $G$-conjugacy classes  represented in the special classes $\mathcal C$  and the rows are indexed  by the irreducible characters of $G$. For the computation we use the fact that aside from the first row, no character of $G$ appears with multiplicity other than $\pm 1$ or $0$.  The  {\sc{Magma}} programme is stored on the ArXiV.

When  calculating $B$, we note that negating a column results in a further solution to the inner product equations. Similarly a permutation of the columns  results in a different candidate for the matrix $B$. These operations correspond to renaming the characters in $C$ or to negating a character in $C$ and so we may consider the solutions to be equivalent. We always arrange that the first row corresponds to the principal character. The calculation reveals   four  solutions, whose rows   are inequivalent under the manipulations just described. The columns of the fragments corresponding to the elements $x$, $y$ and $z$  are presented in Table~\ref{frag}. In fragments one and two, the rows consisting of zeros have non-zero entries on other members of the special classes and so we have left the rows in our partial tables. Remember that these fragments now have rows that may be the negative of character values.

\begin{table}[h]\begin{tabular}{cccc}
\begin{tabular}{|ccc|}\hline
$x$&$z$&$y$\\\hline
1&1&1\\
4&1&-2\\
-3&3&0\\
0&-3&3\\
3&3&3\\
4&1&-2\\
-4&2&-1\\
0&0&0\\
0&3&-3\\
0&0&0\\
-5&1&-2\\
-4&-1&2\\
0&-6&-3\\
\vdots&\vdots&\vdots\\
&&\\
\hline
\end{tabular} &\begin{tabular}{|ccc|}\hline
$x$&$z$&$y$\\\hline
1&1&1\\
4&1&-2\\
-3&3&0\\
0&-3&3\\
-4&2&-1\\
-3&-6&0\\
4&-2&1\\
0&0&0\\
0&-3&-6\\
4&-2&1\\
-5&-2&1\\
0&0&0\\\vdots&\vdots&\vdots\\&&\\&&\\\hline
\end{tabular}&
\begin{tabular}{|ccc|}\hline $x$&$z$&$y$\\\hline
1&1&1\\
3&0&-3\\
1&4&-2\\
3&-3&0\\
-3&-3&-3\\
-4&-1&2\\
4&-2&1\\
3&-3&0\\
0&3&-3\\
3&3&3\\
2&2&2\\
-3&3&0\\
-4&-1&2\\\vdots&\vdots&\vdots\\&&\\\hline
\end{tabular}&
\begin{tabular}{|ccc|}\hline $x$&$z$&$y$\\\hline
1&1&1\\
4&1&-2\\
0&3&-3\\
3&-3&0\\
-4&-1&2\\
-3&-3&-3\\
-1&-1&-1\\
-3&3&0\\
0&3&-3\\
-3&-3&-3\\
-3&3&0\\
2&2&2\\
-3&3&0\\
-4&-1&2\\\vdots&\vdots&\vdots\\\hline
\end{tabular}\end{tabular}
\caption{The candidates for inequivalent fragments of the character table of $G$.}\label{frag}
\end{table}

For the candidates for the partial character table of $G$ given in Table~\ref{frag}, we show that the first three cases are not associated with a character table of a group which satisfies Lemmas~\ref{index}, \ref{mod9} or \ref{mod27} whereas  in the fourth case we demonstrate that $G= H$. All the calculations  make use of the structure constants presented in Lemma~\ref{structconst}.
We denote the virtual characters of $G$ represented in the fragments in Table~\ref{frag}  by $\theta_1, \dots, \theta_s$ where $s$ is the number of rows in the corresponding fragment and $\theta_1$ is the principal character of $G$. For $1 \le i \le s$,  define $$d_i = \theta_i(1).$$  Thus $d_i$ is an integer and is positive if and only if $\theta_i$ is a character.
We set $$g = \frac{|G|}{54^2}= \frac 2 9 |G:H|.$$
\begin{lemma}\label{frag1}
Fragment  one of Table~\ref{frag} is not possible.
\end{lemma}

\begin{proof}
The values of the structure constants $a_{xyz}=3$, $a_{xxy}=2$ and $a_{yyx}=4$  given in Lemma~\ref{structconst} yield, respectively,  the following equalities:
\begin{eqnarray*}
\frac{6}{g}&=&1 -\frac{8}{d_2}+\frac{27}{d_5}-\frac{8}{d_6}+\frac{8}{d_7}+\frac{10}{d_{11}}+\frac{8}{d_{12}}\\
\frac{8}{g}&=&1- \frac{32}{d_2}+\frac{27}{d_5}-\frac{32}{d_6}-\frac{16}{d_7}-\frac{50}{d_{11}}+\frac{32}{d_{12}}\\
\frac{4}{g}&=&1+\frac{16}{d_2}+\frac{27}{d_5}+\frac{16}{d_6}-\frac{4}{d_7}-\frac{20}{d_{11}}-\frac{16}{d_{12}}.
\end{eqnarray*}
To simplify our notation we set $w_1= -\frac{8}{d_2}-\frac{8}{d_6}+\frac{8}{d_{12}}$, $w_2 = -\frac{4}{d_7}$, $w_3 = \frac{10}{d_{11}}$ and $w_4 = 1+\frac{27}{d_5}$.
\begin{eqnarray}
\label{a}\frac{6}{g}-w_4&=&w_1-2w_2+w_3\\
\label{b}\frac{8}{g}-w_4&=&4w_1+4w_2-5w_3\\
\label{c}\frac{4}{g}-w_4&=&-2w_1+w_2-2w_3.
\end{eqnarray}
Subtracting four times Eqn. \ref{a} from Eqn. \ref{b} and adding two times Eqn. \ref{a} to Eqn. \ref{c} we get
\begin{eqnarray*}
\label{d}-\frac{16}{g}+3w_4&=&12w_2-9w_3\\
\label{e}\frac{16}{g}-3w_4&=&-3w_2
\end{eqnarray*}
which means that $w_2 = w_3$. Thus $-4d_{11}= 10 d_7$.  Now, by Lemma~\ref{mod9},  $d_{11} \equiv 1 \pmod 9$ and $d_7 \equiv 2 \pmod 9$. This implies  $-4 d_{11} \equiv -4 \pmod 9$ and $10 d_7 \equiv 20 \pmod9= 2 \pmod 9$, which is a contradiction.
\end{proof}

\begin{lemma}\label{frag2}
Fragment  two of Table~\ref{frag} is not possible.
\end{lemma}

\begin{proof}We consider the structure constants  $a_{xyz}=3$  and $a_{yyx}=4$. These provide the equations
 \begin{eqnarray}
\label{dd}\frac{6}{g}&=&1 -\frac{8}{d_2}+\frac{8}{d_5}-\frac{8}{d_7}-\frac{8}{d_{10}}+\frac{10}{d_{11}}\\
\label{xx}\frac{4}{g}&=&1+\frac{16}{d_2}-\frac{4}{d_5}+\frac{4}{d_7}+\frac{4}{d_{10}}-\frac{5}{d_{11}}.
\end{eqnarray}

Adding two times Eqn.\ref{xx} to Eqn. \ref{dd} gives  the conclusion $$\frac{14}{g}=3+\frac{24}{d_2}.$$ This  means that $7\cdot 9d_{2} -3d_2|G:H|= 24|G:H|$  and   yields  $$7\cdot 3d_{2} -d_2|G:H|= 8|G:H|.$$ By Lemma \ref{mod9} we have $d_2 \equiv 1 \pmod 9$ and  by Lemma \ref{index} $|G:H| \equiv 1 \pmod 9$. Therefore
$$2 \equiv 7 \cdot 3d_2 - d_2|G:H| = 8|G:H| \equiv -1 \pmod 9,$$
which is absurd. Thus Fragment two does not lead to a group satisfying the current assumptions.
 \end{proof}

\begin{lemma}\label{frag3}
Fragment  three of Table~\ref{frag} is not possible.
\end{lemma}

 \begin{proof}  In this case we use all the  structure constants  other than  $a_{xyz}$ to reach our conclusion.
 The structure constants $a_{xxy}=2$, $a_{yyx}=4$, $a_{xxz}=0$,  $a_{zzx}=4$, $a_{yyz}=6$, and $a_{zzy}=2$ lead, respectively, to the following six equations:
\begin{eqnarray*}
\frac{8}{g} &= &1 - \frac{27}{d_2} - \frac{2}{d_3} - \frac{27}{d_5} + \frac{32}{d_6} + \frac{16}{d_7} + \frac{27}{d_{10}} + \frac{8}{d_{11}} + \frac{32}{d_{13}}\\
\frac{4}{g} &=&  1 + \frac{27}{d_2} + \frac{4}{d_3} - \frac{27}{d_5} - \frac{16}{d_6} + \frac{4}{d_7} + \frac{27}{d_{10}} + \frac{8}{d_{11}} - \frac{16}{d_{13}}\\
0 &=& 1 + \frac{4}{d_3} -\frac{27}{d_4} - \frac{27}{d_5} - \frac{16}{d_6} - \frac{32}{d_7} - \frac{27}{d_8} + \frac{27}{d_{10}} +\frac{ 8}{d_{11}} + \frac{27}{d_{12}} - \frac{16}{d_{13}}\\
\frac{9}{g} &=& 1 + \frac{16}{d_3} + \frac{27}{d_4} - \frac{27}{d_5} - \frac{4}{d_6} + \frac{16}{d_7} + \frac{27}{d_8} + \frac{27}{d_{10}} +\frac{8}{d_{11}}- \frac{27}{d_{12}} - \frac{4}{d_{13}}\\
\frac{6}{g} &=& 1 + \frac{16}{d_3} - \frac{27}{d_5} - \frac{4}{d_6} - \frac{2}{d_7} + \frac{27}{d_9} + \frac{27}{d_{10}} + \frac{8}{d_{11}} - \frac{4}{d_{13}}\\
\frac{9}{2g} &=& 1 - \frac{32}{d_3} - \frac{27}{d_5} + \frac{2}{d_6} + \frac{4}{d_7} - \frac{27}{d_9} + \frac{27}{d_{10}} + \frac{8}{d_{11}} + \frac{2}{d_{13}}.
\end{eqnarray*}

 We set
 $w_1 = \frac{27}{d_2}$, $w_2= \frac{2}{d_3}$,  $w_3=-\frac{27}{d_4}-\frac{27}{d_8}+\frac{27}{d_{12}}$, $w_4=1-\frac{27}{d_5}+\frac{27}{d_{10}}+\frac{8}{d_{11}}$,  $w_5= \frac{2}{d_6}+\frac{2}{d_{13}}$, $w_6 = \frac{2}{d_7}$ and $w_7 = \frac{27}{d_9}$.

 \begin{eqnarray}
\frac{8}{g}-w_4&=&-w_1-w_2+16w_5+8w_6\\
 \frac{4}{g}-w_4&=&w_1+2w_2-8w_5+2w_6\\
 -w_4 &=& 2w_2+w_3 -8w_5-16w_6\\
\frac{ 9}{g}-w_4&=&8w_2-w_3-2w_5+8w_6\\
 \frac{6}{g}-w_4&=&8w_2-2w_5-w_6+w_7\\
\frac{ 9}{2g}-w_4&=&-16w_2+w_5+2w_6-w_7.
\end{eqnarray}
If we add the first four equations and subtract twice the sum of the last two equations,  we obtain
$$0 = 27w_2.$$
Therefore $0 = w_2 = \frac{2}{d_3}$, which is ridiculous.
\end{proof}

\begin{lemma}\label{frag4}
Fragment  four of Table~\ref{frag} leads to the conclusion $G=H$.
\end{lemma}

\begin{proof} This time we focus on the structure constants $a_{xyz}=3$, $a_{xxy}=2$ and $a_{yyz}=6$. These, respectively, provide the following equations:
\begin{eqnarray*}
\frac{6}{g}& =& 1 - \frac{8}{d_2}+\frac{8}{d_5}-\frac{27}{d_6} - \frac{1}{d_7} - \frac{27}{d_{10}} + \frac{8}{d_{12}} + \frac{8}{d_{14}}\\
 \frac{8}{g} &=& 1- \frac{32}{d_2} + \frac{32}{d_5} - \frac{27}{d_6}   - \frac{1}{d_7} - \frac{27}{d_{10}} +\frac{8}{d_{12}} + \frac{32}{d_{14}}\\
\frac{6}{g} &=& 1 + \frac{4}{d_2} + \frac{27}{d_3} - \frac{4}{d_5 } - \frac{27}{d_6} - \frac{1}{d_7} + \frac{27}{d_9} - \frac{27}{d_{10}} + \frac{8}{d_{12}} - \frac{4}{d_{14}}.
\end{eqnarray*}

 Thus setting  $w_1 = 1  -\frac{27}{d_6} - \frac{1}{d_7} -\frac{27}{d_{10}} + \frac{8}{d_{12}}$, $w_2= -\frac{4}{d_2}+\frac{4}{d_5}+\frac{4}{d_{14}}$ and $w_3= \frac{27}{d_3}+\frac{27}{d_9}$, we see that
\begin{eqnarray*}
\frac{6}{g}-w_1&=&2w_2\\
\frac{8}{g}-w_1&=&8w_2\\
\frac{ 6}{g}-w_1&=&-w_2+w_3.
\end{eqnarray*}
Therefore $w_3= 1/g$ and so
\begin{eqnarray*}
\frac{27}{d_3} + \frac{27}{d_9} = \frac{9}{2}|G:H|^{-1}.
\end{eqnarray*}
Hence
\begin{eqnarray}
\label{1} d_3d_9 &=&6|G:H|(d_3+d_9).
\end{eqnarray}
 In particular $d_3 + d_9 \not= 0$.
As $d_3^2$ and $d_9^2$ are both less than $|G|$,   $|G|>|d_3d_9| $ and consequently Eqn.~\ref{1} also implies $|G| >  6|G:H||d_3+d_9|$. This means that $$|d_3+d_9| < |H|/6 = 108.$$  On the other hand, \begin{eqnarray*}0 &< &(d_3+d_9)^2 = d_3^2+d_9^2 +2d_3d_9\\& = & d_3^2+d_9^2 +12|G:H|(d_3+d_9) < |G| +12|G:H|(d_3+d_9)\end{eqnarray*} and so $$|H|+ 12(d_3+d_9)>0. $$We conclude  $d_3+d_9> -54$.   Furthermore, if $d_3+d_9$ is positive, then, by Eqn. \ref{1}, $d_3$ and $d_9$ are both positive and, as $(d_3 + d_9)^2 = d_3^2 + d_9^2 + 12|G:H|(d_3+d_9)$, we see $$12|G:H| \le d_3+d_9 \le 108$$  which gives $|G:H|\le 9$. Since Lemma~\ref{index} states $|G:H| \equiv 1 \pmod {27}$, this means that $G=H$ as desired. Therefore we may assume that $d_3+d_9< 0$.

By Lemma~\ref{mod27},  $d_3 \equiv d_9\equiv 12 \pmod {27}$. Hence  $a= d_3/3$ and $b = d_9/3$ are integers  with $a \equiv b \equiv 4 \pmod 9$. As $d_3 + d_9 > -54$,  we have  $-18<a+b<0$.   Additionally Eqn. \ref{1} becomes
\begin{eqnarray}
\label{2} ab &=&2|G:H|(a+b).
\end{eqnarray}
We now determine the possibilities for $a$ and $b$ modulo $27$. From Lemma~\ref{index} we have $|G:H| \equiv 1 \pmod{27}$ and therefore Eqn. ~\ref{2} yields $ab \equiv 2(a+b) \pmod {27}. $  Since $a \equiv b \equiv 4  \pmod 9$, we have that $a$ and $b$ are equivalent to one of $4, 13$ or $22$ modulo 27. As  $ab \equiv 2(a+b) \pmod {27} $, we infer that either $a \equiv b \equiv 4 \pmod {27}$, or up to change of notation, $a\equiv 13 \pmod {27}$ and $b \equiv 22 \pmod 9$. In particular $a+b \equiv 8 \pmod{27}$.  Since $-18 <a+b < 0$,  this is impossible.
\end{proof}

\begin{proof}[The proof of Theorem\ref{G=H}] Since fragments  1, 2 and 3 of Table~\ref{frag} are not associated to a possible character table of $G$ and since fragment 4 leads to the conclusion that $G=H$, we must have $G=H$.
\end{proof}

\section{The proof of Theorem~\ref{McL}}\label{typ}

In this section we assume the hypothesis of the Theorem~\ref{McL} which we recall reads as follows.

\begin{hyp}\label{McLHyp}   G is a finite group, $S \in \Syl_3
(G)$, $Z = Z(S)$ and $J$ is an elementary
abelian subgroup of $S$ of order $3^4$.  Furthermore
\begin{enumerate}
\item  $O^{3'}(N_G(Z)) \approx 3^{1+4}_+.2\udot\Alt(5)$;
\item  $O^{3^\prime}(N_G(J)) \approx 3^4.\Alt(6)$; and
\item $C_G(O_3(C_G(Z)))\le O_3(C_G(Z))$.
\end{enumerate}
\end{hyp}

  We use the same notation as established in
\cite[Section 5]{McL}. Thus $$Q= O_3(N_G(Z))=O_3(C_G(Z))$$ is extraspecial of order $3^5$ and exponent $3$,   $$L= N_G(Z),\;  L_*= O^{3'}(L), \;M= N_G(J) \text{ and } M_* = O^{3'}(M).$$ Notice that $$J= O_3(M)= O_3(M^*).$$With this notation Hypothesis~\ref{McLHyp}  (i) and (ii) are expressed as  $$L_*/Q \cong
2\udot \Alt(5) \cong \SL_2(5)\text { and } M_*/J \cong \Alt(6)\cong \Omega_4^-(3).$$

\begin{lemma}\label{basic} The following hold:
\begin{enumerate}
\item $C_G(Q)= Z(Q) = Z$ has order $3$;
\item $C_G(J)=J$;
\item $J= J(S)$ is the Thompson subgroup of $S$; and
\item $S=JQ$ and  $N_G(S) =L \cap M$.
\end{enumerate}
\end{lemma}

\begin{proof}  As $C_G(Q) \le Q$ by Hypothesis~\ref{McLHyp} (iii) and $Q$ is extraspecial, part  (i) holds.  Since $S \le L$, $|S:Q|= 3$ and $Q$ is extraspecial of order $3^5$,   $J$ is a maximal abelian subgroup of $S$ and  $Z \le J$. It follows  from the described structure of $L$ and $M$ that $C_G(J) \le J$ and, as $J$ is abelian, (ii) holds.

For (iii) suppose there is $\hat{J} \le S$ with $\hat{J}$ abelian and $|\hat{J}| \ge |J|$. Then, as $S/J \in \syl_3(M_*/J) $ and $M_*/J\cong \Alt(6)$,  $|J \cap \hat{J}| \ge 3^2$.  If $S=J\hat J$, then $Z(S) \ge J\cap \hat J $ which contradicts (i). Thus $J\hat{J} \not= S$ and so $|J \cap \hat{J}| \geq 3^3$. If $\hat{J} \not= J$, then, by  Hypothesis~\ref{McLHyp} (ii),  $J\hat{J}$ is not normal in $N_{N_G(J)}(S)$. Hence there is $x \in N_G(S)$ such that $S = J\hat{J}\hat{J}^x$. Further we have $|J \cap \hat{J}| = |J \cap \hat{J}^x| = 3^3$, in particular $|J \cap \hat{J} \cap \hat{J}^x| \geq 3^2$. Since this group  is contained in $Z(S)$, we  again have a contradiction to (i). Hence $J= \hat J$ and $J=J(S)$.  This is (iii).

Since $J \not \le Q$ and $|S/Q|=3$, we have $S= JQ$. Also  $N _G(S) \le N_G(Z(S))=L $ and $N _G(S) \le N_G(J(S))=N_G(J)=M$. Hence $N_G(S) \le L \cap M$. Finally we have $L \cap M \le N_G(S)$  as $S=JQ$. Thus (iv) holds.
\end{proof}

\begin{lemma}\label{M_*mod} Let $X= M^*/J \cong \Alt(6)$.  Then, as a $\GF(3)X$-module, $J$ can be identified with the irreducible 4-dimensional section
of the natural 6-point $\GF(3)X$-permutation module for $\Alt(6)$. In particular, $J$ supports a non-degenerate
orthogonal form which is invariant under the action of $M$.
\end{lemma}
\begin{proof} See \cite[Lemma 5.4 and the discussion at the bottom of page 1769]{McL}.
\end{proof}

\begin{lemma}\label{Modprops} Let $X = \Alt(6)$ and $V$ be the irreducible $4$-dimensional section of the natural 6-point $\GF(3)X$-permutation module. Then
\begin{enumerate}
\item $X$ has three orbits on the $1$-dimensional subspaces of $V$. One orbit has length  $10$ and the other two orbits both have length $15$.
\item If $\langle \nu \rangle \le V$ is in an $X$-orbit of length $15$, then $C_X(\nu) \cong \Alt(4)$.
\item Every hyperplane of $V$ contains an element from the orbit of length 10.
\item If $x \in X$ is of order 4, then $C_V(x) = 0$ and $|C_V(x^2)| = 3^2$.
\end{enumerate}
\end{lemma}

\begin{proof} These statements are the results of  easy  calculations.\end{proof}

  Set $$L_0 = L_*N_{M_*}(S) \text{ and } M_0 = M_*N_{L_*}(S).$$ Since, by Lemma~\ref{basic} (iv),  $N_G(S) =M \cap L$, $L_0$ and $M_0$ are subgroups of $G$.

In the next lemma $2^-\Sym(5)$ denotes the double cover of $\Sym(5)$ which contains $2\udot \Alt(5)$ and in which the transpositions of $\Sym(5)$ lift to elements of order $4$. The Sylow $2$-subgroups of $2^-\Sym(5)$ are quaternion groups of order $16$.

\begin{lemma}\label{M0}
The following hold. \begin{enumerate} \item $M_0/J\cong  \Mat(10)$, $L_0/Q\cong  2^-\Sym(5)$ and
$N_{L_*}(S)N_{M_*}(S)$ has Sylow 2-subgroups which are isomorphic to $\Q_8$.
\item $|L : L_0| = |M :M_0| \le 2$. \item If $|L : L_0| = 2$, then $M/J\cong  2\times \Mat(10)$ and $L/Q \cong
(4\circ \SL_2(5)).2$. Furthermore, $N_G(S)$ has Sylow 2-subgroups which are isomorphic to $2\times
\Q_8$.\end{enumerate}
\end{lemma}
\begin{proof} See \cite[Lemma 5.11]{McL}.\end{proof}

 \begin{lemma}\label{det1} All involutions contained in $M$ act with determinant 1 on $J$ and project to elements in $F^*(M/J)$.
\end{lemma}

\begin{proof}
 By Lemma~\ref{M0} either $M/J \cong \Mat(10)$ or $M/J \cong \Mat(10)\times 2$.   As all involutions of $\Mat(10)$ are contained in $\Alt(6)$, all the involutions of $M/J$ are contained in $F^*(M/J)$.
 If $F^*(M/J) \cong \Alt(6)$, then, as $\Alt(6)$ is perfect, the result holds. If $F^*(M/J) \cong \Alt(6) \times 2$, then the central involution in $M/J$ inverts $J$ and so also has determinant 1. This completes the proof.
 \end{proof}

\begin{lemma}\label{centralizer9} Suppose that $A \le M$ has order $4$  and  $|C_J(A)| \geq 9$. Then $C_J(A)$ contains a conjugate of $Z$.
\end{lemma}

\begin{proof} As any involution of $\Mat(10)$ is contained in $\Alt(6)$ and, by  Lemma~\ref{M0}, $|M:M_0|\le 2$ and $M_0/J \cong \Mat(10)$,   $A$ contains an involution $a\in M_*$. By  conjugating $A$ by elements of $M$, we may assume that $a$ normalizes $S$. Because $N_{M_*}(S)/S$ is cyclic of order $4$,  $a$ is a square in $N_{M_*}(S)$. Hence $a$ centralizes $Z(S) = Z$. As $a$ has determinant 1 in its action on $J$, we have that $|C_J(a)| = 9$ and so $C_J(a) = C_J(A)$ contains $Z$.
\end{proof}

Since, by Lemma~\ref{basic} (iii), $J$ is the Thompson subgroup of $S$ and because $J$ is abelian, \cite[37.6]{Aschbacher} implies $M$ controls $G$-fusion of elements in $J$. Therefore Lemma~\ref{Modprops} (i) implies that there are at most three and at least two conjugacy classes of subgroups of order $3$ in $J$.  We know that the $3$-central class is represented by $Z$ and that the non-trivial elements of $Z$ correspond to singular vectors in $J$. As $\Alt(6)$ is not isomorphic to a subgroup of $\Omega_4^+(3)$ (which is soluble), the quadratic form on $J$ which is preserved up to similarity by $M$ is of $-$-type. Thus there are no subgroups of $J$ of order $9$ in which all the subgroups of order $3$ are conjugate to $Z$.

\begin{lemma}\label{Mfusion}
\begin{enumerate}
\item  $M$ controls $G$-fusion   in $J$;
\item $M$  has exactly
two orbits when acting on the subgroups of order $3$ in $J$; and    \item $Z$ is weakly closed in $Q$ with respect to $G$.\end{enumerate}
\end{lemma}

\begin{proof} We have already discussed (i).

As $M_0/J\cong \Mat(10)$ by Lemma~\ref{M0} and  $\Mat(10)$ has no subgroups of index $15$, we deduce that $\langle y \rangle ^{M_0}$ has size $30$ and therefore
 $C_{M_0}(y) = C_{M_*}(y)$ and $C_{M_*}(y)/J \cong \Alt(4)$ by Lemma~\ref{Modprops}(ii). In particular,  there are at most  two orbits of $M_0$ on subgroups of $J$ of order three. Since there are at least two orbits, (ii)  holds.

Suppose that   $Y \le Q\cap J$ has  order $3$ with $Y \neq Z$ and that $Y$ is $G$-conjugate to $Z$. Then $W = YZ$  is a subgroup of $J$ of order $9$ in which every proper subgroup is conjugate to $Z$.   Since $J$ has no such subgroups of order $9$, we infer that no such $Y$ exists. Thus, if $y \in (J\cap Q)\setminus Z$, then $y$ is not $3$-central in $G$. Suppose now that $Y \in Z^G\setminus\{Z\}$ and $Y \le Q$. Then $C_Q(Y) \cong 3 \times 3^{1+2}_+$ and so, as $L^g/Q^g $ has cyclic Sylow $3$-subgroups, $Z= C_Q(Y)' \le Q^g$.  Now $C_{Q^g}(Z)$ normalizes $Q$ and by conjugating in $L$, we may assume that $C_{Q^g}(Z) \le S$.  But then $C_{Q^g}(Z)$ normalizes $J$ and consequently, as $S/J$ is abelian, $Y = C_{Q^g}(Z)' \le J \cap Q $, which is a contradiction.
\end{proof}

\begin{lemma}\label{signalizer1}  The only  $3'$-subgroup of $G$ which is normalized by $J$ is the trivial subgroup.
\end{lemma}

\begin{proof} Assume that $J$ normalizes   a non-trivial $3'$-subgroup $X$ of $G$. Then,
as every subgroup  of $J$ of order $27$ contains a conjugate of $Z$ by Lemma~\ref{Modprops} (iii),  as $J$ acts
coprimely on $X$, and, as $X=\langle C_X(J_1)\mid |J:J_1|=3\rangle$, we may assume that $Y= C_{X}(Z)\neq 1$.  But then $Y$ is a non-trivial $3'$-subgroup of $L$. As
$Y$ is normalized by $A = J \cap Q$ and $Y$ normalizes $Q$, $[A,Y] \le Q \cap Y=1$ and hence, as $Y \neq 1$ and  $A$ is a maximal abelian subgroup of $Q$, $Y
\le C_L(A) = J $. But then $Y=1$, which is a contradiction.
\end{proof}

We now select and fix $y \in (Q\cap J)\setminus Z$. By Lemma~\ref{Mfusion} (iii), $y$ is not a $3$-central element of $G$. Define  $$K= C_M(y)\text { and } H= C_G(y).$$ The aim of the  remainder of this section  is to prove that
 $$C_G(O_3(H)) \leq O_3(H).$$

\begin{lemma}\label{CLy} We have $C_L(y) \le K$.\end{lemma}

\begin{proof}  We consider the subgroup $U= \langle y,Z\rangle$ and calculate $C_{L_*}(U)= C_{L_*}(y)$. As $U \le J$ and $J$ is abelian, $J\le C_{L_*}(U)$  and, as $Q$ is extraspecial, $C_Q(U)$  has index $3$ in $Q$ and $ C_Q(U)J$ is a Sylow $3$-subgroup of $C_L(U)$. Since the elements of order $5$ in $L_*$ act fixed-point-freely on $Q/Z$ and since the involutions in $ L_* $ invert $Q/Z$, we infer that $C_{L_*}(U)= C_Q(U)J$. Hence $C_{L_*}(y)$ and $C_L(y)$ are $3$-closed. It follows that $C_L(y) \le N_G(C_Q(U)J)\le N_G(J)=M$ as $J= J(S)$ by Lemma~\ref{basic} (iii).
\end{proof}

Define  $$K_a= \langle (1,2,3), (1,4,7)(2,5,8)(3,6,9), (1,2)(4,5)  \rangle$$ and \begin{eqnarray*}K_b&=&
N_{\Alt(9)}(\langle (1,2,3), (4,5,6), (7,8,9)\rangle)\\&=&  \langle (1,2,3),   (1,4,7)(2,5,8)(3,6,9), (1,2)(4,5),
(1,4)(2,5)(3,6)(7,8)\rangle.\end{eqnarray*}

We have that $K_a$ is isomorphic to a semidirect product of an elementary abelian group of order 27 by $\Alt(4)$ and $K_b$ is isomorphic to a semidirect product of an elementary abelian group of order 27 by $\Sym(4)$. Moreover $|K_b:K_a|=2$.  Note that any elementary abelian subgroup of $\Sym(9)$ of order $27$ is conjugate in $\Sym(9)$ to $\langle (1,2,3), (4,5,6), (7,8,9)\rangle$.

\begin{lemma}\label{K/y}  One of the following holds:
\begin{enumerate}
\item $M/J \cong \Mat(10)$ and $K/\langle y\rangle \cong K_a$; or
\item $M/J \cong \Mat(10)\times 2$ and $K/\langle y \rangle \cong K_b$.
\end{enumerate}
Moreover, $C_{K}(O_3(K)) \le O_3(K)$.
\end{lemma}

\begin{proof} We saw in the proof of Lemma~\ref{Mfusion} that  $C_{M_0}(y) = C_{M_*}(y)$ and $C_{M_*}(y)/J \cong \Alt(4)$.
Let
 $ A \in \Syl_2(C_{M_*}(y))$ and
 $a \in A^\#$. Then the action of $C_{M_*}(y)$ on the cosets of $[J,a]A\langle y \rangle$ gives an embedding of
 $C_{M_*}(y)/\langle y\rangle$ into $\Sym(9)$. Since $C_{M_*}(y)/\langle y \rangle$ is generated by elements of
 order $3$ and normalizes $J/\langle y \rangle$, we have $C_{M_*}(y)/\langle y \rangle\cong K_a$. Thus,   if
 $M=M_0$, then (i) holds.

 Suppose that $M>M_0$. Then $M/J \cong \Mat(10)\times 2$. Let $b \in M$ be an involution such that $b J\in
 Z(M/J)$. Then $b$ inverts $y$ and also  $y$ is inverted in  $N_{M_*}(C_{M_*}(y))\approx 3^4{:}\Sym(4)$. Thus the diagonal subgroup of $\langle
 b\rangle N_{M_*}(C_{M_*}(y))/J \cong 2 \times \Sym(4)$ which is isomorphic to $\Sym(4)$ centralizes  $y$. Now
 every element of ${M_*}\langle b\rangle$ acts on $J$ with determinant $1$. It follows that every element of $C_{{M_*}\langle
 b\rangle}(y)$ acts on $J/\langle y\rangle$ with determinant $1$. Let $B \in \syl_2(C_{{M_*}\langle
 b\rangle}(y))$ and $t\in Z(B).$ Then $C_{{M_*}\langle
 b\rangle}(y)$  acts on the nine cosets of $[J,t]B\langle y\rangle$  in $ C_{{M_*}\langle
 b\rangle}(y)$. Thus again we have $C_{{M_*}\langle
 b\rangle}(y)/\langle y\rangle $ is isomorphic to a subgroup of $\Sym(9)$ which normalizes a subgroup of order
 $3^3$. Since every element of $C_{{M_*}\langle
 b\rangle}(y)$ acts on $J/\langle y\rangle$ with determinant $1$, we deduce that they have commutator of order  $3^2$ in $J/\langle y \rangle$;  in particular  $C_{{M_*}\langle
 b\rangle}(y)$  does not contain elements conjugate in $\Sym(9)$ to $(1,2)$ or $(1,2)(3,4)(5,6)$. Hence $C_{{M_*}\langle
 b\rangle}(y)/\langle y \rangle$  is isomorphic to a subgroup of $\Alt(9)$. Therefore $C_{M}(y)/\langle y \rangle \cong K_b$.

 Finally, as $O_{3'}(K) =1$ by Lemma~\ref{signalizer1} and $K$ is soluble, $C_K(O_3(K)) \le O_3(K)$.
\end{proof}


%
%

%

 The proof of Lemma~\ref{centralizerH} uses the following theorem of Smith and Tyrer.

\begin{theorem}[Smith-Tyrer \cite{SmTy}]\label{SmT}  Let $D$ be a finite group and let $P$ be a Sylow $p$-subgroup of $D$ for some odd prime $p$. Suppose $P$ is abelian but not cyclic. If $|N_D(P) : C_D(P)| = 2$, then $O^p(D) < D$ or $D$ is $p$-soluble of $p$-length $1$.\qed
\end{theorem}

The next lemma is required when we apply Theorems~\ref{G=H} and \ref{108}.
\begin{lemma}\label{centralizerH} Let $x \in J\setminus \langle y \rangle$. Then $C_{H/\langle y \rangle }(x\langle y \rangle) \leq K/\langle y \rangle$.
\end{lemma}

\begin{proof} Set $U = \langle x,y \rangle$ and let $W$ be the preimage in $H$ of $C_{H/\langle y \rangle }(x\langle y \rangle)$. Then $|U| = 9$, $W$ normalizes $U$ and $O^3(W)$ centralizes $U$.

Assume first that $U$ contains a $G$-conjugate of $Z$. If this is a $K$-conjugate of $Z$, we may assume that $U = \langle Z, y \rangle$. Thus $C_G(U) = C_L(y) \leq K$ by Lemmas~\ref{Mfusion} (iii) and ~\ref{CLy}.

Assume  that $U$ contains a conjugate  $Z^g$ of $Z$, which is not conjugate to $Z$ in $K$. Then $U$ contains exactly two conjugates of $Z$ and $W= C_G(U)$.
 Moreover $(Z^g)^K$ is of length three or six, depending on whether $K/\langle y\rangle \cong K_a$ or $K_b$ respectively. In particular $J$ is a Sylow 3-subgroup of $W$. Now  $W \leq L^g$. As $J$ is a Sylow 3-subgroup of $W$, we have that $U \not\leq Q^g$ and so $Q^gU = Q^g\langle y \rangle$ is a Sylow 3-subgroup of $K^g$. We now have that $C_{K^g}(y) \leq N_{K_g}( Q^gU) $.  Since  $J \leq Q^gU$, we have that $J = J(Q^gU)$ and so $$W \le N_G( J(Q^gU)) \le N_G(J) =M.$$ But then $W \le K$ and we are done.

So  we finally consider the case when $U^\#$ just consists of $G$-conjugates of $y$. As the centre of $C_S(y)= JC_Q(y)$ is equal to $\langle y,Z \rangle$, we again  see that $J$ is a Sylow 3-subgroup of $C_G(U)$. In particular, by the Frattini Argument  $W = (W \cap M) C_G(U)$. Thus it suffices to show that $C_G(U) \le M$.   By Lemma \ref{centralizer9},  a Sylow 2-subgroup of $C_M(U)$ has  order at most 2.

So we have

\begin{claim}\label{clm}
\begin{enumerate}
\item $J/U$ is a Sylow 3-subgroup of $C_G(U)/U$,
\item $|N_{C_G(U)/U}(J/U) : J/U| \leq 2$.
\end{enumerate}
\end{claim}
 Assume that $C_G(U)$ is not $3$-soluble. Then  Burnside's normal $p$-complement Theorem \cite[Theorem 7.4.3]{Gor} and \ref{clm} (i) and (ii) imply $|N_{C_G(U)/U}(J/U) : J/U|  = 2$. Therefore Theorem~\ref{SmT} shows    $O^3(C_G(U)/U) < C_G(U)/U$. However,  there is an involution $a$ from $C_M(U)$ acting on $J/U$ and, by Lemma \ref{det1},  it acts on $J$ with determinant 1. Thus  $a$  inverts $J/U$. In particular $J/U$ is contained in $O^3(C_G(U)/U)$, which is a contradiction.

So we have shown that $C_G(U)/U$ is 3-soluble. Using  Lemma \ref{signalizer1} yields  $O_{3^\prime}(C_G(U)/U) = 1$ and so $J/U$ is normal in  $C_G(U)/U$. But then $C_G(U) \leq M$ and this proves the lemma.
\end{proof}

\begin{lemma}\label{fusionK}
Let $\rho \in K$ be an element of order three. Then $\rho$ is $H$-conjugate  to an element of  $J$ if and only if  $\rho \in J$.
In particular, $J$ is strongly closed in $C_S(y)$ with respect to $H$.
\end{lemma}

\begin{proof}    Assume that  $\rho \in C_S(y) \setminus J$.  As $C_S(y)/\langle y \rangle \cong 3\wr 3$ by Lemma~\ref{K/y},  all elements of order three in the coset $J \rho$ are conjugate into $\langle y \rangle \rho$.  As $C_Q(y) \not\leq J$, we may assume that $\rho \in Q$. So, again using Lemma~\ref{K/y}, we have $C_K(\rho) = \langle y, Z, \rho \rangle \leq Q$. As $Z$ is weakly closed in $Q$ by Lemma~\ref{Mfusion} (iii), we deduce  $N_H(C_K(\rho)) \leq N_G(Z)=L$. As $H \cap L \le K$ by Lemma~\ref{CLy},   $C_K(\rho)$ is a Sylow 3-subgroup of $C_{H}(\rho)$. In particular $\rho$ is not conjugate to any element of $J$. This proves the lemma.
\end{proof}

 Let $$K_c = K_a'=\langle (1,2,3),
(4,5,6), (7,8,9), (1,2)(4,5),(1,2)(7,8) \rangle.$$

\begin{lemma}\label{F} Assume that $M/J \cong \Mat(10)$. Then $H$ has a normal subgroup $F$ of index $3$ and $(F\cap K)/\langle y \rangle \cong K_c$.
\end{lemma}
\begin{proof} By Lemma~\ref{K/y} (i),  $K/J \cong  \Alt(4)$. Notice that as $Z(C_S(y))= \langle y,Z \rangle$,   $C_S(y)\in \syl_3(H)$ and $N_H(C_S(y)) \le K$ by Lemma~\ref{CLy}.
In particular $(N_{H}(C_S(y)))^\prime \leq J$.
Furthermore,  for any Sylow 3-subgroup $P$  of $H$,   $P^\prime$ is contained in a $H$-conjugate of $J$. Hence by Lemma \ref{fusionK}, the focal subgroup $$\langle( N_{H}(C_S(y)))^\prime, C_S(y) \cap P^\prime\mid  P \in \Syl_3(H) \rangle \leq J.$$ Thus  Gr\"un's Theorem \cite[Theorem~7.4.2]{Gor} implies $O^3(H/\langle y \rangle) < H/\langle y \rangle$. The action of $K_a$ on $J$ shows that $J \leq O^3(H)\langle y \rangle$. Hence there is a subgroup $F$ containing $J$, which is of index 3 in $H$ and $(F\cap K)/\langle y \rangle \cong K_c$.
\end{proof}

%

The next theorem is the final step before we  achieve our  goal.

\begin{theorem}\label{a4} Suppose that  $G$ satisfies Hypothesis~\ref{McLHyp}. Then,  for all $j \in J^\#$, $$C_G(O_3(C_G(j))) \leq O_3(C_G(j)).$$
\end{theorem}

\begin{proof}  By Lemma~\ref{Mfusion} (ii),  $j$ is either conjugate to an element of $Z$ or to an element of $\langle y\rangle$. In the former case we have $C_G(O_3(C_G(j))) \leq O_3(C_G(j))$ by Hypothesis~\ref{McLHyp} (iii). Thus we may suppose  $j \in \langle y \rangle$. We distinguish between the two possibilities in Lemma~\ref{M0}.

Suppose first that $M/J \cong \Mat(10)$. By Lemma~\ref{F} we have that $H$ possesses a normal subgroup $F$ of index 3.  By Lemmas~\ref{F} and \ref{centralizerH},   $F/\langle y \rangle $  satisfies the assumption of Theorem \ref{108} (with $G=F/\langle y \rangle$ and $H= (F\cap K)/ \langle y\rangle$). Hence  $F/\langle y\rangle = (F \cap K)/\langle y\rangle$ and so $H= K$. Thus the result follows from Lemma~\ref{K/y}.

Now suppose that $M/J \cong 2 \times \Mat(10)$. Then  Lemmas~\ref{K/y} (ii), \ref{centralizerH}  and \ref{fusionK} imply that $H/\langle y \rangle$ satisfies the assumptions of  Theorem \ref{G=H}. Hence again we get $H=K$ and the result follows from Lemma~\ref{K/y}.
\end{proof}

\begin{proof}[Proof of Theorem~\ref{McL}] Hypothesis~\ref{McLHyp} and Theorem~\ref{a4} provide the hypothesis for   \cite[Theorem~1.1]{McL}. Thus application of \cite[Theorem~1.1]{McL} yields $G \cong \McL$ or $G \cong \Aut(\McL)$ and proves the theorem.\end{proof}

\end{document}